\documentclass{article}
\usepackage[utf8]{inputenc}
\usepackage[dvips]{graphicx}
 \usepackage{amssymb,amsmath}
 \usepackage{amsthm}
\usepackage{epic}
\usepackage{pstricks}
\usepackage{pst-node}
\usepackage{color}

\usepackage{tikz}
\usetikzlibrary{arrows,decorations.markings}
\usepackage[utf8]{inputenc}
\usepackage{pstricks-add}

\newtheorem{theorem}{Theorem}

\newtheorem{lemma}[theorem]{Lemma}
\newtheorem{corollary}[theorem]{Corollary}

 \newcommand{\ovd}[1]{\overrightarrow{#1}}

\begin{document}

\title{Chen and Chvátal's Conjecture in tournaments\footnote{
Support by Basal program AFB170001 and Fondeyct 1180994, Conicyt, {PAPIIT-M\'exico  IN107218 and IN106318, CONACyT-M\'exico 282280 and UNAM-CIC `La Conjetura de Chen-Chv\'atal en Gr\'aficas Dirigidas'.}}}
\author{Gabriela Araujo-Pardo \\
Instituto de Matematicas, \\
Universidad Nacional Aut\'onoma de M\'exico, M\'exico
\\
and
\\
Martín Matamala \\
Depto. de Ingeniería Matemática and \\
Centro de Modelamiento Matemático,\\ 
CNRS-UMI 2807, Universidad de Chile, Chile}

\maketitle

\begin{abstract}
In this work we present a version of the so called
Chen and Chvátal's conjecture for directed graphs.
A line of a directed graph $D$ is 
defined by an ordered pair $(u,v)$, with $u$ and $v$ two distinct vertices
of $D$, as the set of all vertices $w$ such that
$u,v,w$ belong to a shortest directed path in $D$
containing a shortest directed path from $u$ to $v$.

A line is empty if there is no directed path 
from $u$ to $v$. Another option is that a line 
is the set of all vertices. 
The version of the Chen and Chvátal's conjecture we study 
states that if none of previous options hold,
then the number of distinct lines in $D$ is at least its number of vertices.

Our main result is that any tournament satisfies this conjecture as well as
any orientation of a complete bipartite graph of diameter three.
\end{abstract}



\section{Introduction}

The Chen and Chvátal's conjecture introduced in \cite{CC} 
expresses a trade-off 
that may occur in metric spaces similar to the 
well-known result saying that 
every set of $n$ points in the Euclidean plane determines at least $n$ distinct lines unless they are in the same line (see \cite{dbe}).

Let $(V,d)$ be a metric space. Given two distinct points $x,y\in V$,
we say that a point $z\in V$ is between $x$ and $y$ when
$d(x,y)=d(x,z)+d(z,y)$. We denote the set of all points between $x$ and $y$ by $[x,y]$.

The line generated by two distinct points $x,y\in V$ is the
following set of points
\begin{equation}\label{e:deflinegen}
\overline{xy}:=\{z\in V: x\in [z,y], y \in [x,z] {\rm \  or\  } z\in [x,y]\}.
\end{equation}

Chen and Chvátal conjectured that in any metric space
with $n$ points, $n\geq 2$, either the entire space is a line or there are at least $n$
distinct lines (\cite{CC}). In \cite{metricSpace}, it was proved that 
a metric space with $n$ points without universal lines contains $\Omega(\sqrt n)$ different lines.

For metric spaces in which all distances are integral and are at most $k$, 
previous bound can be improved to $\Omega( n/(5k))$, for each $k\geq 3$. 
In \cite{ChCh11, Chvatal2} it was proved that the conjecture holds
for metric spaces with distances $0,1$ or $2$.

A family of metric spaces with integer distances arises from 
metric space induced by graphs. Here the points are the vertices of the graph and the distance between two vertices is defined by the length of a shortest path between them.
In~\cite{AK} and~\cite{BBCCCCFZ}, it was proved that Chen and Chvátal's 
Conjecture holds for metric spaces induced by chordal graphs and for distance-hereditary graphs, respectively. 
The previous results were extended in \cite{AMRZ} to any graph $G$ such that every induced subgraph of $G$ is either a chordal
graph, has a cut-vertex or a non-trivial module.

The purpose of this work is to present a version of this conjecture for directed graphs, hereinafter referred to as digraphs, and to prove that it holds for tournaments 
and for orientations of complete bipartite graphs with oriented diameter at most three.

For digraphs the function defined by shortest dipaths
does not necessarily defines a distance. However, 
for directed graph we can still denote by $[x,y]$ the set of all vertices $z$ such that there is a shortest dipath from $x$ to $y$ containing $z$. Then, Equation \ref{e:deflinegen}
still defines the line generated by two distinct vertices $x$ and $y$ in a digraph.

Here we emphasize that our definition of line inherites 
the oriented character of digraphs: it is not true
in general that $\ovd{xy}=\ovd{yx}$. Moreover, 
contrary to what happens for metric space where a line $\overline{xy}$ always contains $x$ and $y$, and then it is not empty, in a directed graph, if there is no (shortest) dipaths from $x$ to $y$, then $\ovd{xy}$ is empty. Since, this latter situation happens if and only if the digraph is not strongly connected 
we restrict our study to this latter class of digraphs.

In fact, when considering strongly connected digraphs instead of graphs, we are moving from the context of finite metric spaces to that of finite \emph{quasimetric} spaces where symmetry is not required.
Though in this work we focus on directed graphs, it is worth to mention that one can define lines in any quasimetric space $(V,d)$, simply by (re)defining the set $[u,v]$ as the set of all points satisfying $d(u,v)=d(u,z)+d(z,v)$. In this context, an analogous to Chen and Chvátal's conjecture has the same form as the original one: if no line is universal, then there are at least as many lines as vertices. Most results obtained so far for finite metric spaces in the context of Chen and Chvátal conjecture rely heavily in the symmetry of the distance function giving little hope to
extend them easily to the quasimetric framework.
One exception could be the lower bound $\Omega(\sqrt{n})$ proved in  \cite{metricSpace}, Theorem 3.1. Its proof does not use 
the symmetry of the distance function but use Lemma 2.2 which does use it. However, the statement of Lemma 2.2 translate to a meaningful question for a quasimetric space $(V,d)$ without universal line: given $t$ points $x_1,\ldots,x_t\in V$ such that $x_i\in [x_{i-1},x_{i+1}]$, for each $i\in \{2,\ldots,t-1\}$, is it true that $(V,d)$ has at least $t$ different lines?

Chen and Chvátal's conjecture is trivial for bipartite graphs, simple for complete
graphs and it has been solved asymptotically, 
for graphs with bounded diameter. It has also been solved
for graphs of diameter at most two but requiring much
more effort. To the best of our knowledge it is still open 
for graphs of diameter at most three (see \cite{Chvatal2018} for 
a recent survey).

\subsection*{Our contribution}
Here we prove that tournaments satisfies Chen and Chvátal's conjecture. Surprinsingly, the proof we found is much more involved than that for complete graphs. We also prove this conjecture for orientations of complete bipartite graphs
with oriented diameter at most three. 

Let $D$ be a strongly connected graph. 
Let $L(D)$ be the set of all lines defined by distinct pairs of vertices of $D$.
In the next property we present a criteria that allows us to focus our analysis in 
digraphs with minimum in and out-degree at least two.

\begin{lemma}\label{l:degrees}
Let $D$ be a strongly connected digraph with minimum in-degree or minimum out-degree at most one. Then $V\in L(D)$.
\end{lemma}
\begin{proof}  We only consider the case where $D$ has a vertex $y$ with in-degree one; the case when $y$ has out-degree at one is analogous.

Let $xy\in A$. Since $(x,y)$ is a dipath from $x$ to $y$ of length one, it is the shortest $(x,y)$-dipath. So, $x,y\in \ovd{xy}$.
 For $z\in V$, $z\neq y$, there is a dipath from $z$ to $y$ and, since the  in-degree of $y$ is one, this dipath ends with the arc $(x,y)$. 
 Hence, $[zxy]$ and then $z\in \ovd{xy}$. This proves that $V\in L(D)$.
 \end{proof}

 The proof of Lemma \ref{l:degrees} shows that if $y$ has in-degree
 one then for $xy\in A$, the line $\ovd{xy}$ is universal.
 Similarly, if $x$ has out-degree one and $xy\in A$, then the line
 $\ovd{xy}$ is universal as well. The converse in not true. 
 By instance, take a two dicycles $xyz$ and $x'y'z'$ connected
 by the arcs $xy',x'y,yz',y'z,zx',z'x$. Then 
 each vertex has in-degree and out-degree two and
 each arc defines a universal line.

The following corollary is an immediate consequence of Lemma \ref{l:degrees}.

\begin{corollary}\label{c:lessthanfour}
 Let $D=(V,A)$ be a strongly connected orientation of a graph $G$ with minimum degree at most three.
 Then, $V\in L(D)$.
 \end{corollary}
\begin{proof} Let $v$ a vertex of $G$ with $d_G(v)\leq 3$. As $d^+_D(v)+d^-_D(v)=d_G(v)$ we get 
that $d^+(v)\leq 1$ or $d^-(v)\leq 1$ and from Lemma \ref{l:degrees} we obtain the conclusion.
\end{proof}

\section{Lines defined by an arc}

For a strongly connected digraph $D=(V,A)$ and two vertices $u,v\in V$,
we denote by $d(u,v)$ the length of a shortest dipath (directed path) 
from $u$ to $v$. Even though $d(\cdot,\cdot)$ is not a distance, 
since $d(u,v)$ and $d(v,u)$ may differ, it satisfies $d(u,v)=0$ if and only if $u=v$, and the triangle inequality $d(u,v)\leq d(u,w)+d(w,v)$ for all $u,v,w\in V$. 

In this section we study some properties of lines defined
by arcs, that is lines defined by two vertices
$x$ and $y$ such that $xy\in A$. 
We first notice that if $uv\in A$, then $x\in \ovd{uv}$ 
if and only if $u\in [x,v]$ or $v\in [u,x]$.
Hence, if $x\in \ovd{uv}$, then $d(x,u)<d(x,v)$ or $d(v,x)<d(u,x)$.

\begin{lemma}\label{l:orientedbasic}
Let $D=(V,A)$ an oriented graph.
 Let $uv\in A$ and $x\in V\setminus\{u,v\}$. 
 \begin{enumerate}
  \item[] (I) If $vx,xu\in A$, then $x\in \ovd{uv}$.
  \item[] (O) If $d(x,v)\leq d(x,u)$ and $d(u,x)\leq d(v,x)$, then 
$x\notin \ovd{uv}$. Hence, if $xv,ux\in A$, then $x\notin \ovd{uv}$.
  \item[] (R) If $ux,vx\in A$, then  $x\in \ovd{uv}$ implies that $d(x,v)\geq 3$.
  \item[] (L) If $xu,xv\in A$, then  $x\in \ovd{uv}$ implies that $d(u,x)\geq 3$.
 \end{enumerate}
\end{lemma}
\begin{proof}
 In the first case, since $vx\in A$ and $D$ is an oriented graph 
 we get that $d(x,v)\geq 2$. But since $xuv$ is a dipath in $D$ we
 get that it is a shortest path from $x$ to $v$ containing $u$. Hence,
 $x\in \ovd{uv}$.
 
 In the second situation, $ux\in A$ implies that $v\notin [u,x]$.
 Similarly, $xv\in A$ implies that $u\notin [x,v]$ does not hold. Then, $x\notin \ovd{uv}$.
 
 In the third case, as $x\in \ovd{uv}$ either $u\in [x,v]$ or $v\in [u,x]$.
 But, as before, $ux\in A$ implies that $v\notin [u,x]$. Hence,
 $u\in [x,v]$ holds. Since $D$ is an oriented digraph and $ux\in A$ we have
 that $d(x,u)\geq 2$ and then we get the conclusion $d(x,v)\geq 3$.
 
 The last case is similar to the previous one. Now, the only possibility
 for $x$ to be in $\ovd{uv}$ is that $v\in [u,x]$ holds. Since $D$ 
 is an oriented digraph, under the assumption
 $xv\in A$ we get that $d(v,x)\geq 2$ and then $d(u,x)\geq 3$.
 \end{proof}

We shall use extensively the properties proved in Lemma \ref{l:orientedbasic}.
To ease the presentation we shall refer to them just as cases (I), (O), (R) or (L).

Let exemplify their use in the digraph $T_5=(V,A)$ in Figure \ref{f:t5}. 
It is a tournament on five vertices. 

Hence, given $uv\in A$ and $w\neq u,v$, 
exactly one of the hypothesis of the cases (I), (O), (R) or (L) holds.
Moreover, since all the distance are at most two,
case (I) holds if and only if $w\in \ovd{uv}$.
It can be seen easily that the vertices appearing to the right
of line $\ovd{uv}$ in the middle table of Figure \ref{f:t5} do satisfy case (I).

By case (O) we have $x'\notin \ovd{ax}$, $y'\notin \ovd{ya}$ and
$x'\notin \ovd{xy}$.
By case (R) we have $x\notin \ovd{ax'}$, $x'\notin \ovd{y'a}$ and
$y'\notin \ovd{xy}$, $y\notin \ovd{x'x}$.
By case (L) we have $y'\notin \ovd{ax'}$, $y\notin \ovd{y'a}$ and
$a\notin \ovd{x'x}$.

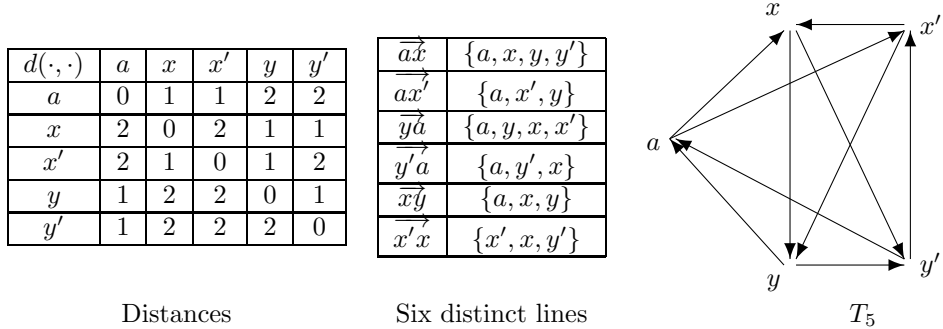
\begin{figure}[ht]
\begin{tabular}{ccc}
\begin{tabular}{|c|c|c|c|c|c|}
\hline
  $d(\cdot,\cdot)$  & $a$ & $x$ & $x'$ & $y$ & $y'$\\ \hline
$a$ & 0 & 1 & 1 & 2 & 2 \\ \hline 
$x$ & 2  &0 & 2 & 1& 1\\ \hline
$x'$& 2  & 1& 0 & 1& 2\\ \hline
$y$ & 1  & 2& 2 & 0 & 1\\ \hline
$y'$& 1  & 2& 2 & 2  & 0\\ \hline
\end{tabular}
&
\begin{tabular}{|c|c|}
\hline
$\ovd{ax}$ & $\{a,x,y,y'\}$ \\ \hline
$\ovd{ax'}$ & $\{a,x',y\}$ \\ \hline
$\ovd{ya}$ &  $\{a,y,x,x'\}$\\ \hline
$\ovd{y'a}$ & $\{a,y',x\}$ \\ \hline 
$\ovd{xy}$ & $\{a,x,y\}$ \\ \hline
$\ovd{x'x}$ & $\{x',x,y'\}$ \\ \hline
\end{tabular}
&
\begin{minipage}{6cm}
\begin{tikzpicture}[scale=0.8,
 decoration={markings, mark=at position 1 with {\arrow[scale=1.5, black]{latex}};}
]
\draw[postaction={decorate}] (0,2.1) -- (1.9,3.9);
\draw[postaction={decorate}] (0,2.1) -- (3.9,3.9);
\draw[postaction={decorate}] (1.9,0) -- (0,2.1);
\draw[postaction={decorate}] (3.9,0.1) -- (0.1,2.1);

\draw[postaction={decorate}] (3.9,4)--(2.1,4);
\draw[postaction={decorate}] (2,3.9)--(2,0.1);
\draw[postaction={decorate}] (2.1,0)--(3.9,0);
\draw[postaction={decorate}] (4,0.1)--(4,3.9);

\draw[postaction={decorate}] (2.1,3.9)--(3.9,0.1);
\draw[postaction={decorate}] (3.9,3.9)--(2.1,0.1);

\node[left] at (0,2) {$a$};
\node[above left] at (2,4) {$x$};
\node[right] at (4,4) {$x'$};
\node[below left] at (2,0) {$y$};
\node[right] at (4,0) {$y'$};

\end{tikzpicture}
\end{minipage}
\\
Distances & Six distinct lines & $T_5$
\end{tabular}
\caption{Use of Lemma \ref{l:orientedbasic}\label{f:t5}.
}

\end{figure}

\begin{lemma}\label{l:lessthansix} 
Every strongly connected oriented graph $D$ with at most five vertices satisfies $V\in L(D)$ or $|L(D)|\geq |D|$.
\end{lemma}
\begin{proof}
 From previous discussion and Corollary \ref{c:lessthanfour}
 we only need to consider strongly connected oriented graph with five vertices.
 From Lemma \ref{l:degrees} we can assume that 
 for each $v\in V$, $d^+(v),d^-(v)\geq 2$ and then, since
 $d^+(v)+d^-(v)\leq 4$ that $d^+(v),d^-(v)=2$. 
 It is easy to see that up to isomorphism $T_5$ is the unique
 oriented graph with this property. And, we have already shown 
 that $T_5$ has at least six distinct lines.
\end{proof}

\section{Tournaments}

To ease the presentation we use the following notation
for $a\in V$, 
$$a^+=\{x: ax\in A\} \textrm{ and }a^-=\{y: ya\in A\}.$$

For $a\in V$, $B\subseteq a^+$  and $C\subseteq a^-$ we define
the sets $(a,B)$ and $(C,a)$ as follows.
$$(a,B)=\{\ovd{az}: z\in B\} \textrm{ and }(C,a)=\{\ovd{za}: z\in C\}.$$

\begin{lemma}\label{l:sameextreme}
Let $D=(V,A)$ be a tournament.
 Let $a\in V$, $B\subseteq a^+$ and $C\subseteq a^-$.
 Then $|(a,B)|=|B|$ and $|(C,a)|=|C|$.
\end{lemma}
\begin{proof}
Since $D$ is a tournament, for $z,z'\in B$
we can assume that $zz'\in A$. 
Then, case (O) implies that $z\notin \ovd{az'}$.
Similarly, for $z,z'\in C$
we can assume that $zz'\in A$. 
Then, case (O) implies that $z'\notin \ovd{za}$.
\end{proof}

Our main theorem is the following.
\begin{theorem}\label{t:mainT}
 Any strongly connected tournament $D$ has at least $|D|$ distinct lines or a universal line.
\end{theorem}

\subsection*{Simple cases}
From Lemma \ref{l:lessthansix} we can assume that $|V|\geq 6$.
The proof of Theorem \ref{t:mainT} requires yet some additional properties. 
In order to focus on the difficult cases let us consider now the 
situation when for some $a\in V$, $(a,a^+)\cap (a^-,a)=\emptyset$.
Then, from Lemma \ref{l:sameextreme}, $D$ has at least $|a^+|+|a^-|=|V|-1$ distinct lines. 

\begin{lemma}\label{l:norepeatedlines}
 Led $D$ be a strongly connected tournament and $a\in V$ such that 
$(a,a^+)\cap (a^-,a)=\emptyset$. Then $|L(D)|\geq |V|$. 
\end{lemma}
\begin{proof}
In order to obtain a line not in $(a,a^+)\cup (a^-,a)$
we consider two situations.
If there are $x\in a^+$ and $y\in a^-$ such that $yx\in A$, 
then by case (O), $a\notin \ovd{yx}$. Since any line in $(a,a^+)\cup (a^-,a)$
contains $a$, we conclude that $\ovd{yx}$ is not in $(a,a^+)\cup (a^-,a)$.
Therefore, $|L(D)|\geq |V|$.

Otherwise, we can assume that for every $x\in a^+$ and $y\in a^-$ we have that 
$xy\in A$. In this case, $d(x,a)=2$, for all $x\in a^+$.
From Lemma \ref{l:degrees} we also can assume that 
$d^-(x)\geq 2$. Then, there is $x'\in a^+$, $x'\neq x$ and such that $x'x\in A$. 
Since $d(x,a)=2$, from case (R)  we 
get that $a\notin \ovd{x'x}$. As before, this shows that 
there are at least $|V|$ lines in $D$.
\end{proof}
\subsection*{Repeated lines}

In the rest of this section we assume that for all $a\in V$, 
there is $x\in a^+$ and $y\in a^-$ such that $\ovd{ax}=\ovd{ya}$.

\begin{lemma}\label{l:edgexy}
Let $a$ be a vertex of $D$. 
 If there are $x\in a^+$ and $y\in a^-$ such that $\ovd{ax}=\ovd{ya}$
 then $xy\in A$.
\end{lemma}
\begin{proof} If $yx\in A$, then $y\in [x,a]$ because $x\in \ovd{ya}$. Then, 
there is $z$ with $xz\in A$ such that $y\in [z,a]$. Hence,
$za\notin A$ and then $d(a,y)\leq d(x,y)$. Thus, 
$x\notin [a,y]$ and $a\notin [y,x]$ which shows the contradiction $y\notin \ovd{ax}$.
\end{proof}

For the proof of Theorem \ref{t:mainT}
we consider a vertex $a$ of $D$, $x\in a^+$ and $y\in a^-$ such that $\ovd{ax}=\ovd{ya}$.
Let $X,Y$ and $Z$ be the following sets of lines. 
$$X:=(a^-\cap x^-,x)\cup (x,a^+\cap y^-\setminus \{x\}),$$ 
$$Y:=(y,a^+\cap y^+)\cup (a^-\cap x^+\setminus \{y\},y)$$
and
$$Z:=(a,a^+)\cup (a^-,a).$$
We shall prove that in this case $X\cup Y\cup Z\cup \{\ovd{xy}\}$
has at least $|V|$ distinct lines (see Figure \ref{f:schemeproof}).


\newcounter{xmm}
\newcounter{ymm}
\newcounter{xdmm}
\newcounter{ydmm}

\newcommand{\elipse}[4]
{\setcounter{xmm}{#1}
\setcounter{ymm}{#2}
\setcounter{xdmm}{#3}
\setcounter{ydmm}{#4}

\draw (\arabic{xmm},\arabic{ymm}) to [out=90,in=180] 
(\arabic{xmm}+\arabic{xdmm},\arabic{ymm}+\arabic{ydmm}) to [out=0,in=90] 
(\arabic{xmm}+2*\arabic{xdmm},\arabic{ymm}) 
to [out=-90,in=0] 
(\arabic{xmm}+\arabic{xdmm},\arabic{ymm}-\arabic{ydmm})
to [out=180,in=-90] 
(\arabic{xmm},\arabic{ymm});
}
\begin{figure}[ht]
\begin{center}
\begin{tikzpicture}[scale=0.5,
decoration={markings, mark=at position 1 with {\arrow[scale=1.5, black]{latex}};}
]
\elipse{2}{8}{8}{2}
\node at (1,9) {$a^+$};
\elipse{3}{8}{3}{1}
\node at (6,9.5) {$a^+\cap y^-$};
\elipse{10}{8}{3}{1}
\node at (13,9.5) {$a^+\cap y^+$};
\elipse{2}{3}{8}{2}
\node at (1,2) {$a^-$};
\elipse{3}{3}{3}{1}
\node at (6,1.5) {$a^-\cap x^+$};
\elipse{10}{3}{3}{1}
\node at (13,1.5) {$a^-\cap x^-$};
\node at (0,5.5) {$a$};
\draw[postaction={decorate}] (0,6) to [out=80,in=190] (2,8);
\draw[postaction={decorate}] (2,3) to [out=-190,in=-80] (0,5);
\node at (4,8) {$x'$};
\node at (7,8) {$x$};
\node at (12,8) {$x''$};
\node at (4,3) {$y'$};
\node at (7,3) {$y$};
\node at (12,3) {$y''$};
\draw[postaction={decorate}] (7.5,3.5) -- (11,6.8);
\draw[postaction={decorate}] (7,6.8) -- (7,3.5);
\draw[postaction={decorate}] (11,4)--(7.3,7.8);
\draw[postaction={decorate}] (6.6,7.8) -- (4,4.1);
\end{tikzpicture}
\caption{A representation of sets $a^+,a^-, a^+\cap y^-,a^+\cap y^+,
a^-\cap x^+$ and $a^-\cap x^-$, and the roles of vertices $x,x',x'',y,y'$ and $y''$.}
\label{f:schemeproof}
\end{center}
\end{figure}
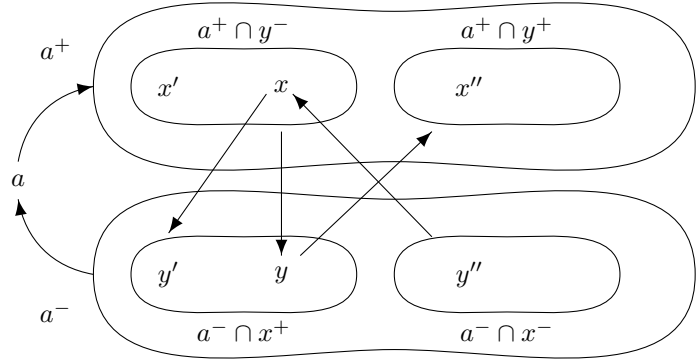

\begin{lemma}\label{l:XYZdisjoint}
 Let $a$ be a vertex of $D$, $x\in a^+$, $y\in a^-$ such that 
$\ovd{ax}=\ovd{ya}$.
Then, $X,Y$ and $Z$ are pairwise disjoint and $\ovd{xy}\notin X\cup Y$.
\end{lemma}
\begin{proof}
Clearly each line in $Z$ contains $a$.
We prove that no line in $X\cup Y$ contains $a$.
From this the last statement follows easily since from 
Lemma \ref{l:edgexy} $xy\in A$ and then $a\in \ovd{xy}$.

By case (O) no line in $(a^-\cap x^-,x) \cup (y,a^+\cap y^+)$  contains $a$.

By case (I), $a^+\cap y^-\subseteq \ovd{ya}$. Since 
$\ovd{ax}=\ovd{ya}$ and case (O) we get that 
$a^+\cap y^-\setminus \{x\}\subseteq x^+$. As 
$d(x',a)=2$, for every $x'\in a^+\cap y^-$ from case (L) 
we get that no line in $(x,a^+\cap y^-\setminus \{x\})$ contains $a$.

Similarly, by case (I), $a^-\cap x^+\subseteq \ovd{ax}$. Since 
$\ovd{ax}=\ovd{ya}$ and case (O) we get that 
$a^-\cap x^+\setminus \{y\}\subseteq y^-$. As 
$d(a,y')=2$, for every $y'\in a^-\cap x^+$ from case (R) 
we get that no line in $(a^-\cap x^+\setminus \{y\},y)$ contains $a$. Therefore, no line in $X\cup Y$ contains $a$.
 
Notice that any line in $X$ contains $x$ and any line in $Y$
contains $y$.

By case (R) no line in $(x,a^+\cap y^-\setminus \{x\})$ contains $y$.
Similarly, by case (L) no line in $(a^-\cap x^+\setminus \{y\},y)$ contains $x$.

Hence, it only remains to show that $(a^-\cap x^-,x)\cap (y,a^+\cap y^+)=\emptyset.$

We have that $d(y,x)=2$. If there are $x''\in a^+\cap y^+$ 
  and $y''\in a^-\cap x^-$ such that $\ovd{yx''}=\ovd{y''x}$,
  then {$y''y\notin A$} since, case (L) and $y\in \ovd{y''x}$
  imply that $d(y,x)\geq 3$. Hence, {$yy''\in A$}.
  As $y''\in \ovd{yx''}$ case (O) implies that $x''y''\in A$,
  but case (R) and $d(y'',x'')=2$ imply the contradiction $y''\notin \ovd{yx''}$.

\end{proof}

\begin{lemma}\label{l:structure}
 Let $a\in V$, $x\in a^+$ and $y\in a^-$ such that 
 $\ovd{ax}=\ovd{ya}$ and $\ovd{xy}\in Z$.
 Then $a^+\cap y^-=\{x\}$ and $a^-\cap x^+=\{y\}$.
\end{lemma}
\begin{proof}
From Lemma \ref{l:edgexy} we know that $xy\in A$.
Let $z\in a^+\setminus \{x\}$. We prove that if 
$x\in \ovd{az}$, then $y\notin \ovd{az}$. 

Since $ax,az\in A$ by case (O) we have that $zx\in A$.
The same case shows that $z\notin \ovd{ax}$. But under 
the assumption $\ovd{ax}=\ovd{ya}$ and again case (O)
we get that $yz\in A$. Case (L) and $d(a,y)=2$ shows
that $y\notin \ovd{az}$. 

In a similar manner we can prove that when 
$z\in a^-\setminus \{y\}$ and $y\in \ovd{za}$, then $x\notin \ovd{za}$. In fact, if $y\in \ovd{za}$, then by case (O) we
get that $yz\in A$ and thus $z\notin \ovd{ya}=\ovd{ax}$.
Hence, $zx\in A$ and, by case (R) and $d(x,a)=2$ we get that 
$x\notin \ovd{za}$.

\end{proof}

In fact with the same idea given in the proof of Lemma \ref{l:structure} it can be proven that when $\ovd{xy}\in Z$, then 
$\ovd{xy}=\ovd{ax}=\ovd{ya}=\{a,x,y\}$.

%
%

 \begin{lemma}\label{l:uniquerepetition}
Let $a$ be a vertex of $D$.
Let $x\in a^+$ and $y\in a^-$ be such that $\ovd{ax}=\ovd{ya}$.
Then, $|(a,a^+) \cap (a^-\cap x^+,a)|, |(a^-,a) \cap (a,a^+\cap y^-)|\leq 1$.
\end{lemma}
\begin{proof}
Let us assume that there are $x'\in a^+\cap y^-$ and $y'\in a^-$
such that $\ovd{ax'}=\ovd{y'a}$. 
Notice that it is enough to prove that $x'=x$ since $\ovd{y'a}=\ovd{ya}$ implies
that $y'=y$, for Lemma \ref{l:sameextreme}.

 From Lemma  \ref{l:edgexy} 
 we know that $xy,x'y'\in A$.
 For the sake of contradiction, let us assume that $x'\neq x$.
From the assumption $x'y\in A$ and case (I) we get that 
 $x'\in \ovd{ya}=\ovd{ax}$ and $y\in \ovd{ax'}=\ovd{y'a}$, 
 and by case (O) that $xx'\in A$ and $yy'\in A$.
By case (O), $y'\notin \ovd{ya}=\ovd{ax}$. By case (I), $y'x\in A$.
Then, $d(x',x)=2$ which by case (R) implies
the contradiction $x'\notin \ovd{ax}$. 
 
When there are $x'\in a^+$ and $y'\in a^-\cap x^+$ 
such that $\ovd{ax'}=\ovd{y'a}$, we proceed in a similar way. 
\end{proof}

We now give the proof of Theorem \ref{t:mainT}.

\begin{proof} (of Theorem \ref{t:mainT})
Let $a$ be any vertex of $D$. From Lemma \ref{l:norepeatedlines} 
we can assume that $(a,a^+)\cap (a^-,a)\neq \emptyset$. Let 
$x\in a^+$ and $y\in a^-$ such that $\ovd{ax}=\ovd{ya}$.
From Lemma \ref{l:edgexy} we have that $xy\in A$.

We know from Lemma \ref{l:XYZdisjoint}
that the sets $X,Y$ and $Z$ are pairwise disjoint. Let $W=X\cup Y\cup Z$ and $\overline{W}=L(D)\setminus W$.

Since 
$$(a,a^+)\cup (a^-\cap x^+,a), (a^-,a)\cup (a,a^+\cap y^-)\subseteq Z,$$
$$(a^-\cap x^-,x),(x,a^+\cap y^-\setminus \{x\})\subseteq X$$ and 
$$(y,a^+\cap y^+), (a^-\cap x^+\setminus \{y\},y)\subseteq Y,$$
the cardinality of $W$ is at least the maximum of the following values.
\begin{itemize}
 \item $|(a,a^+)\cup (a^-\cap x^+,a)|+|(a^-\cap x^-,x)|+|(y,a^+\cap y^+)|$.
\item $|(a,a^+)\cup (a^-\cap x^+,a)|+|(a^-\cap x^-,x)|+
  |(a^-\cap x^+\setminus \{y\},y)|$.
\item $|(a^-,a)\cup (a,a^+\cap y^-)|+|(a^-\cap x^-,x)|+|(y,a^+\cap y^+)|$.
\item $|(a^-,a)\cup (a,a^+\cap y^-)|+|(y,a^+\cap y^+)|+|(x,a^+\cap y^-\setminus \{x\})|$
\end{itemize}
  From Lemma \ref{l:uniquerepetition} we have that $|(a,a^+)\cap (a^-\cap x^+,a)|,   |(a^-,a)\cap (a^+\cap y^-,a)|\leq 1$. Hence, $|W|$ is at least the maximum of the following values:
\begin{itemize}
 \item   $|a^+|+|a^-\cap x^+|-1+|a^-\cap x^-|+|a^+\cap y^+|=|V|+|a^+\cap y^+|-2$, 
 \item   $|a^+|+|a^-\cap x^+|-1+|a^-\cap x^-|+|a^-\cap x^+|-1=|V|+|a^-\cap x^+|-3$,
 \item   $|a^-|+|a^+\cap y^-|-1+|a^-\cap x^-|+|a^+\cap y^+|=|V|+|a^-\cap x^-|-2,$ 
 \item   $||a^-|+|a^+\cap y^-|-1+|a^+\cap y^+|+|a^+\cap y^-|-1=|V|+|a^+\cap y^-|-3$.
\end{itemize}

Then, $$|L(D)|\geq |V|+|\overline{W}|+k-2,$$
where 
$$k=\max\{|a^+\cap y^-|-1,|a^-\cap x^+|-1,|a^+\cap y^+|,|a^-\cap x^-|\}.$$
From Lemma \ref{l:lessthansix} we can assume that $|V|\geq 6$.
As $$V=\{a\}\cup (a^+\cap y^-)\cup (a^+\cap y^+)\cup 
(a^-\cap x^+)\cup (a^-\cap x^-),$$
we get that $k\geq 1$. 
If $\overline{W}$ is not empty, then we get the conclusion $|L(D)|\geq |V|$. 
Otherwise, $\ovd{xy}\in W$. From Lemma \ref{l:XYZdisjoint}
we get that $\ovd{xy}\in Z$ and from Lemma \ref{l:structure} we get that $a^+\cap y^-=\{x\}$ and $a^-\cap x^+=\{y\}$. But, since $|V|\geq 6$, in this latter situation $k\geq 2$.

\end{proof}

\section{Orientation of complete bipartite graphs}

Let $B=(V,A)$ be an orientation of diameter at most three 
of a complete bipartite graph with independent sets $X$ and $Y$.
We prove that $|L(B)|\geq |V|$ or $V\in L(B)$.
If $V\notin L(B)$, then from Lemma \ref{l:degrees} we know
that, for each $u\in V$, $|u^+|,|u^-|\geq 2$ and 
then $|X|,|Y|\geq 4$.

Due to the restriction on the diameter we have that 
$d(u,v)=d(v,u)=2$, when $u$ and $v$ belong to the same 
independent set, and that $d(u,v)\in \{1,3\}$ when 
they are in different independent sets.

Hence, if $u$ and $v$ are in $X$, then 
$$\ovd{uv}=\{u,v\}\cup (u^+\cap v^-)\cup (u^-\cap v^+).$$

Though $\ovd{uv}=\ovd{vu}$, these lines still have a nice property: $\ovd{uv}\cap X=\{u,v\}$. Hence, $B$ has at least $\binom{|X|}{2}$
different lines of this type. Since we can assume that $|X|\geq \lceil |V|/2 \rceil$, we get the result for $|V|\geq 9$.

We now focus in the case when $B$ has 8 vertices. 
We prove that lines defined by arcs are different from the $\binom{4}{2}=6$ previous ones, and that there are at least two different such lines.

Let $uv\in A$.  Then $\ovd{uv}\cap u^+=\{v\}$ and $\ovd{uv}\cap v^-=\{u\}$. In fact, if $z\in u^+$, $z\neq v$, then $d(z,u)=3>d(z,v)=2$ which implies that $u\notin [z,v]$. As $v\notin [u,z]$ we get that $z\notin \ovd{uv}$. A similar argument
shows the second statement. 

From them we get the analogous of Lemma \ref{l:sameextreme}: $|(u,u^+)|=|u^+|$ and $|(v^-,v)|=|v^-|$, and the structure of the line $\ovd{uv}$:  
$$\ovd{uv}=\{u,v\}\cup u^-\cup v^+.$$
Moreover, if $u\in X$, then $|\ovd{uv}\cap X|=|\{u\}\cup v^+|\geq 3$ and $|\ovd{uv}\cap Y|=|\{v\}\cup u^-|\geq 3$.

This last property shows that the lines in $(u,u^+)$ are different from lines defined by two vertices in the same independent set. Since $|(u,u^+)|=|u^+|\geq 2$ we conclude
that $B$ has at least $\binom{6}{2}+2=8$ different lines showing the validity of the following result.
\begin{theorem}\label{t:bipdiam3}
 Orientations of diameter at most three of complete bipartite graphs satisfy Chen and Chvátal's conjecture.
\end{theorem}

\section{Conclusion}

We believe that with similar ideas to those developed in this work it can be proved that strongly connected orientations with oriented diameter at most two have either an universal line or have
at least as many lines as vertices.



\end{document}